\documentclass[a4paper]{article}
\usepackage{amsmath,amsthm,amssymb,bm, comment}
\usepackage{authblk}
\usepackage{xcolor}
\usepackage{graphicx} 
\usepackage{hyperref}
\usepackage{booktabs}
\usepackage{adjustbox}
\usepackage{multirow}
\usepackage[normalem]{ulem}

\usepackage[linesnumbered,ruled,vlined]{algorithm2e}

\DeclareMathOperator{\ord}{ord}

\newcommand{\cA}{\mathcal{A}}
\newcommand{\cB}{\mathcal{B}}

\newcommand{\cD}{\mathcal{D}}
\newcommand{\cE}{\mathcal{E}}
\newcommand{\cF}{\mathcal{F}}
\newcommand{\cG}{\mathcal{G}}
\newcommand{\cH}{\mathcal{H}}

\newcommand{\cO}{\mathcal{O}}

\newcommand{\cW}{\mathcal{W}}
\newcommand{\cM}{\mathcal{M}}
\newcommand{\cZ}{\mathcal{Z}}

\newcommand{\bbZ}{\mathbb{Z}}
\newcommand{\bbC}{\mathbb{C}}

\newcommand{\coC}{\mathbf{C}}
\def\Spec{{\rm Spec}}
\def\dres{\partial{\rm Res}}

\def\sdres{{\rm s}\partial{\rm Res}}
\def\ord{\rm ord}

\newcommand{\gd}{{\rm GD}}
\def\Ker{\rm Ker}

\def\gcrd{\rm gcrd}

\def\gcd{\rm gcd}

\def\BC{\texttt{BC}}

\def\ec{e_{\mathcal{A}}}

\def\rank{{\rm rank}}

\def\0{{\bf 0}}

\newtheorem{thm}{Theorem}[section]

\newtheorem{lem}[thm]{Lemma}

\newtheorem{cor}{Corollary}[thm]

\theoremstyle{definition}

\newtheorem{example}[thm]{Example}

\theoremstyle{remark}
\newtheorem{rem}[thm]{Remark}
\newtheorem*{remark}{Remark}

\newcommand{\centr}{\mathcal{Z}}

\SetKwInOut{Input}{Input}
\SetKwInOut{Output}{Output}
\SetKwComment{Comment}{// }{}

\title
{On the classification of centralizers of ODOs:\\ An effective differential algebra approach.}
\author[a]{Sonia L. Rueda}

\affil[a]{Universidad Politécnica de Madrid}

\date{}

\begin{document}

\maketitle

\begin{abstract}
The correspondence between commutative rings of ordinary differential operators  and algebraic curves has been extensively and deeply studied since the seminal works of Burchnall-Chaundy in 1923. 
This paper is an overview of recent developments to make this correspondence computationally effective, by means of differential algebra and symbolic computation. 

The centralizer of an ordinary differential operator $L$ is a maximal commutative ring, the affine ring of a spectral curve $\Gamma$.
The defining ideal of $\Gamma$ is described, assuming that a finite set of generators of the centralizer is known. In addition, the factorization problem of $L-\lambda$ over a new coefficient field, defined by the coefficient field of $L$ and $\Gamma$, allows the computation of spectral sheaves on $\Gamma$. The ultimate goal is to develop Picard-Vessiot theory for spectral problems $L(y)=\lambda y$, for a generic $\lambda$.
\end{abstract}

\textsc{\href{https://zbmath.org/classification/}{MSC[2020]}:}  
13N10, 13P15, 12H05



\section{Introduction}\label{sec:introduction}

The centralizer $\cZ(L)$ of a nonzero ordinary differential operator (ODO) $L$, in the ring of differential operators, is a commutative algebra.
This result was first proved by I. Schur in 1904 \cite{Sch} and is one of the characteristic
features of ordinary differential operators, with coefficients in a commutative differential ring, with no nonzero nilpotent elements \cite{Good}. 
Schur’s result does not hold for other algebras such as ODOs with matrix coefficients \cite{PRZ2023}  or the algebra of partial differential operators \cite{KasmanPreviato2001}, \cite{Zheglov2020}.

The classification of all commutative algebras of ODOs has been studied by many authors and in diverse context of motivations, including Burchnall and Chaundy \cite{BC1}, Gelfand and Dikii \cite{GD}, Krichever \cite{K77bis}, \cite{K78}, Mumford \cite{Mum2}, Segal–Wilson \cite{SW}, Previato-Wilson \cite{PW}, Verdier \cite{Ve}, Mulase \cite{Mulase} and more recently Zheglov \cite{Zheglov}. For ODOs, whose coefficients are formal power series, a bijective correspondence was established between the following sets of objects in \cite{Mulase}:
The set of equivalence classes of commutative algebras $\cA$ of ordinary differential operators with a monic element, (for an invertible function $\sigma$, the algebras $\cA$ and $\sigma \cA \sigma^{-1}$ are identified); The set of equivalence classes of  isomorphism classes of {\sf quintets} consisting of an algebraic curve and on it a point, a vector bundle, a local covering and a local trivialization. This is achieved via a contravariant functor, generalization of the 
Krichever map, establishing an equivalence between the category of quintets and the category of infinite dimensional vector spaces corresponding
to certain points of Grassmannians together with their stabilizers, the so called {\sf Schur pairs} \cite{Zheglov}. Mulase's work \cite{Mulase} was the last one of a long research history, more details can be found in the introduction of \cite{PW} or \cite{Zheglov}.

Let us consider an ODO $L$ with coefficients in a differential field $\Sigma$, whose field of constants is algebraically closed and of zero characteristic.
Centralizers are affine rings of algebraic curves, $\Spec (\cZ(L))$ and one of our goals is to compute the defining ideal of this abstract algebraic curve using the differential resultant of two differential operators \cite{Cha}. It was proved by E. Previato \cite{Prev} and G. Wilson \cite{Wilson} that the defining polynomial of the planar spectral curve of the algebra $\coC [P,Q]$ generated by two commuting operators is defined by the square-free part $f(\lambda,\mu)$ of the differential resultant
\[h(\lambda,\mu)=\dres(P-\lambda, Q-\mu).\]
Moreover $f$ is a Burchnall-Chaundy (BC) polynomial, since $f$ has constant coefficients and $f(P,Q)=0$.

An algorithm has been designed in \cite{JR2025}, to compute a finite basis of $\cZ (L)$ as a $\coC [L]$-module and to compute families of ODOs  of arbitrary order with non-trivial centralizer $\cZ(L)\neq \coC [L]$, based on K. Goodearl's results \cite{Good} and the computation of the Gelfand-Dickey hierarchies implemented in \cite{JRZHD2025}. We use these bases and differential resultants to approach the computation of the defining ideal of $\Spec (\cZ(L))$. So far it has been achieved successfully in the case of centralizers of ODOs of order $3$ in \cite{RZ2024}. In this paper, we present results to extend our methods to commutative algebras $\cA$ of ODOs, containing two operators of coprime order. These are algebras of rank $1$, being the rank the greatest common divisor of all orders.  ODOs with centralizers of rank $1$ are also called algebro-geometric \cite{We}. In general, it is  non-trivial to determine effectively if an algebra of commuting operators has rank $1$. In many occasions this can be detected through an effective computation of the centralizer \cite{JR2025}, but in some cases, the algorithm in \cite{JR2025} cannot guarantee to compute the whole centralizer but only a maximal subalgebra of rank $r$, being $r$ the rank of an initially known pair of commuting operators.

The differential Galois group of $L-\lambda$ for a generic $\lambda$, \textcolor{teal}{was defined in \cite{BEG} } and the centralizer of $L$ was proved to have rank $1$ if the group is commutative.  The ultimate goal of our research project is to develop the Picard-Vessiot theory of spectral problems $L(y)=\lambda y$ for algebro-geometric ODOs, and give algorithms to compute differential Galois groups. Together with J.J. Morales Ruiz and M.A. Zurro, the spectral Picard-Vessiot field of $L(y)=\lambda y$ was described for an algebro-geometric Schr\" odinger operators $L$ \cite{MRZ2}.  
This problem is linked to a 
factorization problem over a new coefficient field $\Sigma(\Gamma_{\cA})$ defined by the spectral curve $\Gamma_{\cA}$ of the centralizer $\cA=\cZ(L)$, whose field of constants is not algebraically closed.

\medskip

{\bf The paper is organized as follows.} In Section \ref{sec-basis} we review results of K. Goodearl in \cite{Good} to describe the basis of a centralizer $\cZ(L)$ as a $\coC [L]$-module. For operators of prime order,  Theorem \ref{thm-centralizer} describes the basis of the centralizer and some conclusions are derived on the notion of $\rank (\cZ(L))$. In general it is not easy to decide if a centralizer is trivial $\cZ(L)=\coC[L]$, but for differential operators of prime order in the first Weyl algebra, we can give an easy criteria Corollary \ref{cor-cenWeyl}. 

Section \ref{sec-Scurve} is devoted to the description of an isomorphism between $\cA$ and the coordinate ring of the spectral curve $\Gamma_{\cA}$. This construction allows to define the ideal $\BC(\cA)$ of the spectral curve, which is the ideal of all Burchnall-Chaundy polynomials. We review some classical results about the computation of $\BC(\cA)$ using differential resultants in the case of planar spectral curves in Theorem \ref{thm-hPQ} and recent results on the case of space spectral curves in Theorem \ref{thm-BCL}. 

The goal of Section \ref{sec-Ssheaves} is to define a vector bundle $\cF_{\cA}$ on the points of the spectral curve $\Gamma_{\cA}$, and describe its computation using differential subresultants, for planar and space spectral curves in theorems 
\ref{thm-SSPlanar} and \ref{thm-SSSpatial}. This problem is a factorization problem of $L-\lambda$ over the field of the curve $\Sigma (\Gamma_{\cA})$. In the last section, Section \ref{sec-PV}, the spectral Picard-Vessiot field of the spectral problem $L(y)=\lambda y$ is described for operators $L$ of orders $2$ and $3$, emphasizing that they are differential extensions of a non algebraically closed field. 

\medskip

The examples throughout the paper where computed using Maple 2024 and the package \texttt{dalgebra} in SageMath, designed to perform symbolic calculations in differential domains. The algorithms to compute the almost commuting basis and the GD hierarchies in the ring of ODOs, are implemented in the  module \texttt{almost\_commuting.py} ~\cite{JRZHD2025,SAGE2023}. This software is publicly available and can be obtained from the following link:
\begin{center}\url{https://github.com/Antonio-JP/dalgebra/releases/tag/v0.0.5}\end{center}

\medskip

{\bf Notation.} For concepts in differential algebra, we refer to \cite{Kolchin, Ritt} and for Picard-Vessiot theory to \cite{VPS}, \cite{Morales}. 
Let $\Sigma$ be an ordinary differential field with derivation $\partial$, whose field of constants $\coC$ is algebraically closed and of zero characteristic. In this paper we will consider examples with coefficients in the differential ring $\bbC[x]$ of complex coefficient polynomials and the differential field  of rational functions $\bbC (x)$ with derivation $d/dx$, and also in the differential ring $\bbC \{\eta\}$ with $\eta=\cosh(x)$, whose differential field is denoted $\bbC \langle \eta\rangle$.
We will also mention the domain $\bbC[[x]]$ of formal power series with complex number coefficients, which is a differential ring with $d/dx$. We denote by $\bbC((x))$ its quotient field, the field of formal Laurent series.

Let us denote by $\Sigma((\partial^{-1}))$ the ring of pseudo-differential operators in $\partial$ with coefficients in  $\Sigma$
\[
\Sigma((\partial^{-1}))=\left\{\sum_{-\infty< i\leq n} a_i\partial^i\mid a_i\in  R , n\in\bbZ\right\},
\]
where $\partial^{-1}$ is the inverse of $\partial$ in $R((\partial^{-1}))$, $\partial^{-1}\partial=\partial\partial^{-1}=1$.
We denote by $\cD=\Sigma [\partial]$ the non-commutative ring of linear differential operators and recall that its ideals are principal. 
In fact $\cD$ admits Euclidean division. 
Given ~$L$,~$M$ in $\cD$, if $\ord(M)\geq \ord(L)$ then $M=qL+r$ with $\ord(r)<\ord(L)$, $q,r\in \Sigma[\partial]$. Let us denote by $\gcrd(L,M)$ the greatest common (right) divisor of $L$ and $M$.

\section{Basis of centralizers}\label{sec-basis}

Given an ordinary differential operator $L$ in $\Sigma [\partial]$, let us consider the centralizer of $L$ in $\Sigma[\partial]$, 
\begin{equation*}
    \cZ(L)=\{A\in \Sigma [\partial]\mid LA=AL\},
\end{equation*}
in the non-trivial case of operators $L\in \Sigma[\partial]\backslash \coC[\partial]$.
In most cases $\cZ(L)=\coC [L]$, but we will study the special case where $\coC [L]$ is strictly included in $\cZ (L)$ and say that {\sf the centralizer is non-trivial}.

We will study the structure of $\cZ(L)$ as a $\coC[L]$-module, as it will allow us to describe $\cZ(L)$ by a finite basis. For this purpose
we will review next \cite[Theorem 1.2]{Good}.  Let us assume that $\ord(L)=n$ and define the $\coC [L]$-submodules 
\[\cZ_i=\{Q\in\cZ (L)\mid \ord(Q)\equiv_n i\}, i=0,1,\ldots ,n-1.\] 
If $\cZ_i$ is non empty then there exists $A_i$ in $\cZ_i$ of minimal order.
Observe that $\cZ_0=\coC[L]$ and $A_0=1$. By \cite[Theorem 1.2]{Good}, the centralizer $\cZ(L)$ is a $\coC[L]$-module with basis
\begin{equation}\label{eq-basis}
    \cB(L)=\{A_i\mid \cZ_i\neq \emptyset, i=0,1,\ldots ,n-1\}.
\end{equation}
In addition, the next theorem holds, we include its proof for completion, due to its importance for the results in this paper.

\begin{thm}\label{thm-rank}
Given a nonzero differential operator $L$ in $\Sigma[\partial]$, the cardinal of a basis of $\cZ(L)$ as a $\coC[L]$-module divides $\ord(L)$. Moreover
\[\cZ(L)=\oplus_{i\in I} \coC[L] A_i\]
where $I=\{i\in\{1,\ldots ,n-1\}\mid \cZ_i\neq \emptyset\}$.
\end{thm}
\begin{proof}
Observe that the set $\cO=\{\ord(A)+n\bbZ\mid A\in \cZ(L), A\neq 0\}$ is closed under addition, it is a cyclic subgroup of $\bbZ/n\bbZ$. Therefore the cardinal of $\cO$ divides $n$. A bijection between the basis $\cB(L)$ defined in \eqref{eq-basis} and $\cO$ sending $A_i$ to $i+n\bbZ$ is established, proving that the cardinal of $\cB(L)$ divides $\ord(L)$.
\end{proof}

In this setting, where the coefficients of differential operators belong to a differential field,  centralizers are commutative domains, see Section \ref{sec-Scurve}. Furthermore centralizers are maximal commutative $\coC$ subalgebras of $\Sigma [\partial]$. Given a commutative subalgebra $\cA$ of $\Sigma[\partial]$, then there exists a maximal commutative subalgebra $\cM$ of $\Sigma [\partial]$ such that
\[\cA \subseteq \cM=\cZ (A), \forall A\in \cA.\]
Thus, for any $A\in\cA$, we can choose an operator $L$ of minimal order within the basis of $\cM$ as a $\coC [A]$-module to compute $\cM$ as $\cZ(L)$.

\medskip

Let us consider a differential operator $L$ in $\Sigma [\partial]\backslash \coC[\partial]$  with non-trivial centralizer $\cZ(L)\neq \coC[L]$. If we assume that $L$ has prime order $p$ 
then we can be more precise about the structure of $\cZ(L)$ as a $\coC[L]$-module.

\begin{thm}\label{thm-centralizer}
Let $L$ be an operator in $\Sigma [\partial]\backslash \coC[\partial]$ of prime order $p$ with nontrivial centralizer $\cZ(L)\neq \coC[L]$. Then $\cZ(L)$ is a free $\coC[L]$-module and bases are of the form  
\begin{equation}\label{eq-basisp}
   \cB (L)=\{1, A_{1} ,\ldots , A_{p-1}\},
\end{equation}  
with $A_i$ as defined in \eqref{eq-basis}. In addition
\[\cZ (L)=\coC [L,A_1,\ldots ,A_{p-1}].\]
\end{thm}
\begin{proof}
By Theorem \ref{thm-rank} the rank of $\cZ(L)$ as a free $\coC[L]$-module is a divisor of $p$. Since $\cZ(L)\neq \coC[L]$ then the number of elements of the basis of $\cZ(L)$ as a $\coC[L]$-module is $p$. Therefore, by \cite{Good}, Theorem 1.2 and \eqref{eq-basis}, the sets $\cZ_i\neq \emptyset$ for $i=0,1,\ldots ,p-1$. 
\end{proof}

\medskip

The {\sf rank of a commutative set of differential operators} in $\Sigma [\partial]$ is the greatest common divisor of their orders, see \cite{PRZ2019} for a discussion about the rank of a set of differential operators. As a consequence of Theorem \ref{thm-centralizer}, the rank of the centralizer is
\begin{equation}
    \rank (\cZ (L))=\gcd (\ord(L),\ord(A_i)\mid i\in I).
\end{equation}

\begin{cor}\label{cor-AG}
Let $L$ be an operator in $\Sigma [\partial]\backslash \coC[\partial]$ with non-trivial centralizer $\cZ(L)\neq \coC[L]$ and consider a basis $\cB(L)=\{A_i, i\in I\}$ of $\cZ(L)$ as a $\coC[L]$-module. The following are equivalent:
\begin{enumerate}
    \item $\rank(\cZ(L))=1$.
    
    \item There exist operators $P$ and $Q$ in $\{L, A_i\mid i\in I\}$ of coprime order.
\end{enumerate}
\end{cor}

Differential operators satisfying the statements of Corollary \ref{cor-AG} are called {\sf algebro-geometric operators}, see for instance \cite{We} or \cite{Grig}.  

\begin{remark}
If $\ord(L)$ is prime then $\rank(\cZ (L))=1$. Otherwise multiple situations can arise. 
    \begin{enumerate}
        \item If $\ord (L)=4$ then $\cB(L)=\{1,A_1,A_2,A_3\}$ and $\rank(\cZ (L))=1$ or, $\cB(L)=\{1,A_2\}$ and $\rank(\cZ(L))=2$.

        \item If $\ord (L)=6$ then $\cB(L)=\{1,A_1,A_2,A_3,A_4,A_5\}$ and $\rank(\cZ (L))=1$,  $\cB(L)=\{1,A_2, A_4\}$ and $\rank(\cZ(L))=2$, or $\cB(L)=\{1,A_3\}$ and $\rank(\cZ(L))=3$.
    \end{enumerate}
\end{remark}


\subsection{Computing ODOs with non-trivial centralizer}

Given a differential operator $L$ in $\Sigma [\partial]$ it is an open problem to determine algorithmically if $\cZ (L)$ is non-trivial.  Although profound and inspiring theoretical studies about the nature of centralizers have been achieved throughout the twentieth century \cite{K77bis}, \cite{Mum2},\cite{Wilson}, \cite{Mulase},\cite{BZ}, \cite{PZ}, new questions arise to achieve the development of practical algorithms. 
To generate explicit examples of ODOs with non-trivial centralizer one could use different methods see \cite{Zheglov}: Krichever's explicit formulae for centralizers of rank one with smooth spectral curve; Wilson's formulae for centralizers of rank one with rational singular spectral curves; the method of deformation of Tyurin parameters from \cite{KN2}, was used by many authors as for instance in \cite{Mi1},
\cite{Mo2}; using quintets as in \cite{Mulase}, see \cite{PZ}, to obtain differential operators with coefficients in C[[x]]. 
A new method is proposed in \cite{JR2025} for computable differential extensions $\coC\langle \eta \rangle$, where $\eta$ is a transcendental element over $\coC\langle \eta \rangle$, say $\eta=\cosh(x)$ or the Weierstrass $\wp$ function.

\medskip

A crucial reason is that the coefficients of $L$ are solutions of nonlinear differential equations. If we restrict to the case of monic ODOs in normal form (the coefficient of $\partial^{n-1}$ is zero)
\[L=\partial^n+u_{2}\partial^{n-2}+\ldots+u_{n-1}\partial+u_n,\]
the following statements are equivalent \cite{Wilson}, \cite{JR2025}:
\begin{enumerate}
    \item $\cZ(L)$ is non-trivial.

    \item $u_2,\ldots ,u_n$ is a solution of a system of the stationary Gelfand-Dickey (GD) hierarchy.
\end{enumerate}
For instance, if $\ord(L)=2$ then $u_2$ must be a solution of one of the equations of the Korteweg-de Vries (KdV) hierarchy and if $n=3$ then $u_2$ and $u_3$ are solutions of a system of the Bousinesq hierarchy. The solutions of the KdV hierarchy are well known \cite{Ve}. It is harder to find families of solutions for the systems of the {\sf stationary Gelfand-Dickey (GD) hierarchy of $L$}, see \cite{Dikii, DS, Wilson} and references therein. For a fixed $m\notin n\bbZ$,  the {\sf $(n,m)$-system of the Gelfand-Dickey hierarchy} is denoted by
\begin{equation}\label{eq-Hnm}
   \gd_{n,m}: \overline{H}_m+\sum_{j\in J_{n,m-1}} c_{j}  \overline{H}_j= \overline{0}, 
\end{equation}
where $\overline{H}_j(U)$ is the row vector $(H_{0,j}(U),\ldots ,H_{n-2,j}(U))$ and
$\{c_j,\ j\in J_{n,m-1}\}$
is a set of algebraic variables, $J_{n,m-1}=\{0,1,\ldots ,m-1\}\backslash n\bbZ$. The system $\gd_{n,m}$ is obtained from the commutator 
\[[A_m,L]=\sum_k \cH_k \partial^{k},\]
for a formal $L$, where the coefficients $u_2,\ldots ,u_n$ are assumed to be differential variables and a generic operator $A_m$ of order $m$ of the $\coC$-vector space of almost commuting operators $W(L)$, for details see \cite{JRZHD2025}. 
The system $\gd_{n,m}$ is obtained forcing the coefficients of $\partial^{k}$ to be zero
\begin{equation}\label{eq-GD}
    \cH_k = H_{k,m}(U) +\sum_{j\in J_{n,m-1}} c_{j}  H_{k,j}(U) \ , \textrm{for } \  k = 0, \dots , n-2,\,\,\, m\geq 2,
\end{equation}
and in row vector form is \eqref{eq-Hnm}.

\medskip

In \cite{Wilson}, the systems $\gd_{n,m}$ are called the modified KdV hierarchy. They play a crucial role in understanding the generalizations of the theory to other affine Kac-Moody algebras see \cite{DS85}, \cite{Wilson81}.

\medskip

In \cite{JR2025}, we use Goodearl's Theorem and our SageMath implementation of GD hierarchies and almost commuting operators in \cite{JRZHD2025} to compute centralizers and generate families of ODOs with non-trivial centralizer.

\begin{example} 
Consider the order $4$ operator given in ~\cite[Example 22]{PZ}.
The centralizer of the next operator of order $4$ 
\begin{align*}
L_4 =\partial^4- 20 \frac{x^3(x^5-120)}{
(x^5 + 30)^2} \partial^2-3000
\frac{x^2(7x^5-90)}{(x^5 + 30)^3} \partial + 18000
\frac{x(3x^{10}-145x^5 + 450)}{
(x^5 + 30)^4},
\end{align*}
was computed with the algorithm in \cite{JR2025} and 
is $\cZ(L_4)=\coC [L_4,A_1,A_2,A_3]$
with
$\ord(A_2)=6$, $\ord(A_3)=7$ and $\ord(A_1)=9$. So even if $\rank(L_4,A_2)=2$ the centralizer $\cZ(L_4)$ is of rank $1$. Computing the basis of $\cZ(L_4)$ allowed to determine that $L_4$ is algebro-geometric. 
\end{example}

\subsection{In the first Weyl algebra}

We can apply Theorem \ref{thm-centralizer} to obtain conclusions on monic differential operators in the first Weyl algebra $\cW$, the algebra of differential operators with polynomial coefficients $\bbC [x][\partial]$, where $\partial=\frac{d}{dx}$. For an operator in $\cW$, the centralizer $\cZ(L)$ is the centralizer in $\cW$.

\begin{cor}\label{cor-cenWeyl}
    Let $L$ be a monic operator in the first Weyl algebra $\cW$ of prime order $n$. Then either $\cZ(L)=\bbC[L]$ or $\cZ(L)=\bbC[\partial+q(x)]$, with $q(x)\in \bbC[x]$. If $n=2$ the centralizer is $\cZ(L)=\bbC[\partial+q(x)]$ whenever 
    $$L=\partial^2+ (2q(x)+c_1)\partial+q(x)c_1 + q(x)^2 + q'(x) + c_2,\,\,\,  c_1,c_2\in \bbC$$
\end{cor}
\begin{proof}
By Theorem \ref{thm-centralizer}, if $\cZ(L)\neq \bbC[L]$ then $\rank (\cZ(L))=1$. By \cite{MakarLimanov2021} Theorem 1.1, Section 4, Case 1, then $\cZ(L)=\bbC[\partial+q(x)]$, for some $q(x)\in \bbC[x]$. 

Moreover, if $L=\partial^2+v(x)\partial+u(x)$, $u(x),v(x)\in \bbC[x]$ and $A=\partial+q(x)$. By direct computation
$[L,A]=0$ is equivalent to $v'-2q'=0$ and $u''-vq'-q''=0$.
This implies that 
\begin{align*}
    v(x)&=2q(x)+c_1\\
    u(x)&=q(x)c_1 + q(x)^2 + q'(x) + c_2 
\end{align*}
with $c_1,c_2\in \bbC$.
\end{proof}

\begin{example}
    By Corollary \ref{cor-cenWeyl}, the centralizer of $L=\partial^2+x^2$ is $\bbC[L]$, as it is easily verified that it does not commute with an operator of the form $\partial+q(x)$.
\end{example}

Consider operators of order $4$ in the first Weyl algebra, that after a Liouville transformation can be assumed of the form 
\begin{gather}\label{op-weyl}
 L_4= \big( \partial^2 +V(x) \big)^2 +U(x)\partial +W(x),
\end{gather}
with $U(x)$, $V(x)$ and $W(x)$ polynomials in $\bbC [x]$. We proved in \cite{PRZ2019}, Section 6 that given $L_4$ irreducible and such that $\cZ(L_4 )\not= \mathbb{C}[L_4 ]$,
then 
by \cite{PRZ2019}, Lemma 6.4
$$\deg(V )> \max\left\{ \frac{1}{2} \deg(U), \frac{1}{2}\deg(W) \right\},$$
and by \cite{PRZ2019}, Corollary 6.5 
\begin{gather*}\cZ (L_4)={\bbC}[L,B_2],\end{gather*}
for an operator $B_2$ of minimal order $2(2g+1)$, for $g\neq 0$, with spectral curve of genus $g$ defined by $\mu^2=R_{2g+1}(\lambda)$. Furthermore $$\rank(\cZ(L_4))=2.$$

In the first Weyl algebra it is even harder to find  rank~$r$ pairs and
important contributions were made by Grinevich~\cite{Gri}, Mokhov~\cite{Mo2}, Mironov~\cite{Mi1}, Davletshina and Mironov~\cite{DM}, Mironov and Zheglov~\cite{Mi2}, Oganesyan~\cite{Og2016}.

The theory of commuting differential operators in the first Weyl algebra has deep connections with the Jacobian and Dixmier conjectures, see also the connections with Beret's conjecture in \cite{GZ24}. A very recent paper  has announced the Dixmier conjecture to hold in the fist Weyl algebra \cite{Z2024}.

\section{Spectral curves}\label{sec-Scurve}

By a famous theorem of I. Schur in \cite{Sch}, stated for analytic coefficients, given a monic differential operator $L$ in $\Sigma [\partial]$, if $\ord(L)=n$ then the centralizer $\centr((L))$ of $L$ in the ring of pseudo-differential operators is determined by the unique monic $n$-th root $L^{1/n}$ of $L$.
For coefficients in a differential ring $R$ this result was proved in \cite[Theorem~3.1]{Good}
\begin{equation}\label{eq-SchurThm}
    \centr((L))=\left\{ \sum_{-\infty<j\leq m} c_j (L^{1/n})^j\mid c_j\in {\bf C}, m\in\bbZ\right\}.
\end{equation}
As a consequence, $\centr((L))$ is a commutative ring. To review the long history of this theorem, see ~\cite[Sections~3 and~4]{Good}. By  ~\cite[Lemma 1.1]{Good}
$\cZ(L)$ is a domain and by ~\cite[Corollary 4.2]{Good} it is a commutative domain since  
\begin{equation}\label{eq-Schur2}
\centr (L)=\centr ((L))\cap \Sigma [\partial].
\end{equation}
By \eqref{eq-SchurThm} and \eqref{eq-Schur2} the quotient field of $\cZ(L)$ is a function field in one variable,   whose transcendence degree is one.
Moreover any subalgebra $\cA$ of $\cZ (L)$ has Krull dimension one, it is the affine ring of an algebraic curve $\Spec(\cA)$. This fact was first proved by Burchnall and Chaundy \cite{BC1, BC2}, and in greater generality by Krichever \cite{K78}, followed by Mumford \cite{Mum2}, Verdier \cite{Ve}, Mulase \cite{Mulase} and \cite{BZ}, Theorem 1.8 with coefficient in  $\bbC((x))$. In this section, we aim to describe the defining ideal of this abstract curve.

\begin{example}
    Let us consider $\Sigma =\bbC (x)$ and $\partial=\frac{d}{dx}$. It is well known that the centralizer of $L=\partial^2-\frac{2}{x^2}$ is
        \[\centr (L)=\bbC[L,A] \mbox{ with }\ord(A)=3,\]
see for instance \cite{MRZ1}. The centralizer $\centr (L)$ is isomorphic to the coordinate ring of the spectral curve defined by $\mu^2=\lambda^3$, namely
        \[\centr (L)\simeq \frac{\bbC[\lambda,\mu]}{(\mu^2-\lambda^3)}.\]
\end{example}

\subsection{Defining ideals of spectral curves}

Burchnall\textendash Chaundy polynomials were originally defined for pairs of commuting differential operators in \cite{BC1}. 
The Burchnall and Chaundy theorem \cite{BC1}, states that two operators of coprime orders $P$ and $Q$ commute if and only if they satisfy $f(P,Q)=0$ for a constant coefficient polynomial $f(\lambda,\mu)$.

\medskip

We will next define the Burchnall\textendash Chaundy ideal of  a finitely generated commutative $\coC$-algebra $\cA$ in $\Sigma [\partial]$. Without lost of generality we can assume that
\begin{equation}\label{eq-cA}
    \cA=\coC[L, G_1,\ldots ,G_{t-1}]
\end{equation}
with $\ord(L)<\ord(G_i)$ and $G_i\notin \coC[L]$. Observe that $\cA\subseteq \cZ (L)$.

\medskip

Let us consider the polynomial ring 
$$\coC [\lambda,\overline{\mu}]=\coC [\lambda,\mu_1,\ldots , \mu_{t-1}],$$
in algebraic variables $\lambda$ and $\mu_i$, and define the ring homomorphism
\begin{equation}\label{eq-eL}
    \ec: \coC [\lambda,\overline{\mu}] \rightarrow \Sigma[\partial] \mbox{ by } \ec(\lambda)=L,\,\,\,  \ec(\mu_i)=G_i.
\end{equation}
It is a ring homomorphism because the generators $\cG=\{L, G_1,\ldots ,G_{t-1}\}$ commute and they commute with the elements of $\coC$, the field of constants of $\Sigma$.
Given $g\in \coC [\lambda,\overline{\mu}]$ we will denote by $g(L,G_1,\ldots , G_{t-1})$ the image $\ec (g)$ of a polynomial $g\in \coC [\lambda,\overline{\mu}]$.
We define the {\sf BC-ideal of} $\cA$ as 
\begin{equation}\label{def-BCL}
\BC (\cA):=\Ker(\ec)=\{g\in\coC [\lambda,\overline{\mu}]\mid g(L,G_1, \ldots ,G_{t-1})=0\}.    
\end{equation}
We will call the elements of the BC ideal  {\sf BC-polynomials}.
The next isomorphism holds
\[\cA=\coC[L,G_1,\ldots ,G_{t-1}]\simeq\frac{\coC [\lambda,\overline{\mu}]}{\BC(\cA)}.\]
Since we are working with differential operators in $\Sigma[\partial]$, then there are no zero divisors in $\cA$. Thus $\cA$ is an integral domain and $\BC(\cA)$ is a prime ideal. Moreover $\cA$ is the coordinate ring of an algebraic curve whose defining ideal is $\BC (\cA)$.

The {\sf spectral curve of 
$\cA$} is the irreducible algebraic variety of the prime ideal $\BC (\cA)$, that is 
\[\Gamma_{\cA}:=\{\eta\in\coC^t\mid g(\eta)=0, \forall g\in\BC(\cA)\}.\]
Let us denote by $\coC[\Gamma_{\cA}]$ the coordinate ring of the curve
\[\coC[\Gamma_{\cA}]:=\frac{\coC [\lambda,\overline{\mu}]}{\BC(\cA)}.\]
We define {\sf the field of the curve} to be $\coC(\Gamma_{\cA})$ the fraction field of $\coC[\Gamma_{\cA}]$.

\subsection{Planar spectral curves}

Consider the case $\cA=\coC[P,Q]$, for commuting differential operators $P$ and $Q$ in $\Sigma[\partial]$, to avoid meaningless situations we will assume that $P,Q\notin \coC[\partial]$ and both have positive order. In this case we will denote the BC ideal of $\cA$ by $\BC (P,Q)$ and the spectral curve $\Gamma_{\cA}$ by $\Gamma_{P,Q}$, and say it is the {\sf spectral curve of the pair $P, Q$}.

\medskip

For an algebraically closed field of constants $\coC$, the defining ideal $\BC(P,Q)$ of $\Gamma_{P,Q}$ is generated by an irreducible polynomial $f\in \coC[\lambda,\mu]$,
\[\BC(P,Q)=(f).\]
The irreducible algebraic variety of the prime ideal $(f)$ is 
\begin{equation}\label{eq-GammaPQ}
\Gamma_{P,Q}:=\left\{ \ P\in\coC^2\mid f(P)=0\right\},  
\end{equation}
which does not reduce to a  a single point in $\coC^2$ because $P,Q\notin \coC$.
By $\ec$ the isomorphism of commutative domains is established
\[\coC [P,Q]\simeq \frac{\coC[\lambda,\mu]}{\BC(P,Q)}=\frac{\coC[\lambda,\mu]}{(f)}.\]

The computation of the defining polynomial $f$ of the spectral curve $\Gamma_{P,Q}$ can be achieved by means of the differential resultant.
Let us assume that $\ord(P)=n$ and $\ord(Q)=m$, the {\sf differential resultant} of $P-\lambda$ and $Q-\mu$ equals
$$h(\lambda,\mu)=\dres(P-\lambda,Q-\mu):=\det (S_0(P-\lambda,Q-\mu)),$$
where the Sylvester matrix $S_0(P-\lambda,Q-\mu)$ is the coefficient matrix of the extended system of differential operators
\[\Xi_0=\{\partial^{m-1}(P-\lambda),\ldots \partial(P-\lambda), P-\lambda, \partial^{n-1}(Q-\mu), \ldots ,\partial (Q-\mu), Q-\mu\}.\]
Observe that $S_0(P-\lambda,Q-\mu)$ is a squared matrix of size $n+m$ and entries in $\Sigma[\lambda,\mu]$. For the definition and main properties of differential resultants for ODOs see \cite{Cha}, \cite{McW} or \cite{RZ2024}.

It was first proved by G. Wilson in \cite{Wilson}  that $h(\lambda,\mu)$ is a constant coefficient polynomial; see also \cite{Prev}.  
For differential operators with coefficients in an arbitrary differential field, these results were proved in \cite{Zheglov}, Section 5.3, see also \cite{RZ2024}, Theorem 5.4 and \cite{MRZ1}. 

\begin{thm}\label{thm-hPQ} Given commuting differential operators $P$ and $Q$ in {$\Sigma [\partial]\backslash \coC[\partial]$, both of positive order}, then
    \[h(\lambda,\mu)=\dres(P-\lambda,Q-\mu)\in \BC(P,Q).\]
{Moreover  $h= f^ {r}$ for a non zero natural number $r$ and $f$ being the irreducible polynomial such that $\BC(P,Q)=(f)$.}
\end{thm}

\subsection{Space spectral curves}\label{sec-SpaceSC}

Consider $\cA=\coC[L,G_1,G_2]$, with $\ord(L)=n$, $\ord(G_i)=o_i$ and denote
\begin{equation*}
h_i(\lambda,\mu_i):=\dres(L-\lambda,G_i-\mu_i)=\mu_i^n-\lambda^{o_i}+\cdots \in\coC [\lambda ,\mu_i],
\end{equation*}
with square free part $f_i$, the radical of $h_i$.
Recall that $f_i$ are irreducible and $\BC(L,G_i)=(f_i)$, $i=1,2$.
In addition let as denote by $f_3$ the irreducible polynomial in $\coC[\mu_1,\mu_2]$ obtained as the radical of 
\begin{equation*}
    \dres(G_1-\mu_1,G_2-\mu_2)=\mu_2^{o_1}-\mu_1^{o_2}+\cdots.
\end{equation*}
We have $\BC(G_1,G_2)=(f_3)$. In the next chain of ideals in $\coC[\lambda,\mu_1,\mu_2]$
\begin{equation}\label{eq-inclusions}
    (0)\subset  (f_i,f_j)\subseteq (f_1,f_2,f_3) \subseteq \BC(\cA),\,\,\, i,j\in \{1,2,3\}
\end{equation}
it is interesting to determine if the  inclusions are identities.

Observe that $\beta = V(f_1 ) \cap V (f_2 )$ is a space algebraic curve. Including $f_3$ in the ideal allows to select one irreducible component.
There is a chain of algebraic varieties 
\begin{equation*}
 \Gamma_{\cA}:=V(\BC(\cA)) \subseteq   V(f_1,f_2,f_3) \subseteq
 V(f_i,f_j),
\end{equation*}
and the question is to find the conditions for equality to hold.

\begin{thm}[\cite{RZ2024}, Theorem 10]\label{thm-BCL} 
Let $\cA=\cZ(L)=\coC[L,A_1,A_2]$ be  the centralizer of an order $3$ operator in $\Sigma [\partial]\backslash \coC[\partial]$.
Given the irreducible polynomials $f_i$, $i=1,2,3$  such that $\BC(L,A_i )=(f_i)$ and $\BC(A_1,A_2)=(f_3)$, then
\begin{equation}
        \BC(\cA)=(f_1,f_2,f_3).
\end{equation}
\end{thm}

\begin{rem}
    The previous Theorem \ref{thm-BCL} can be extended to $\cA$ determined by $3$ generators, assuming that $\rank (\cA)=1$. For instance, the case $\cA=\cZ(L)$, where $L$ has prime order and $\cZ(L)$ has only three generators, see Example \ref{ex-o5}

    {\bf Open problem.} It is an open problem to compute the defining ideal of the spectral curve $\Gamma_{\cA}$ for a commutative algebra $\cA$ in $\Sigma [\partial]$, with more than $2$ generators and of arbitrary rank. 
\end{rem}

\begin{example}\label{ex-o5}
    Consider the differential field $\Sigma=\coC (x)$ with derivation $\partial=d/dx$ and the following order $5$ differential operator in $\Sigma [\partial]$
    \[L= \partial^5 - \frac{55}{x^2}\partial^3 + \frac{85}{x^3}\partial^2 + \frac{235}{x^4}\partial - \frac{640}{x^5}.\]
    In this case, since $n=5$, Goodearl's basis of $\cZ(L)$ as a $\coC[L]$-module is of the form $\{1, A_1, A_2, A_3, A_4\}$. With the algorithm in \cite{JR2025} the basis was computed checking that $\ord(A_i)\equiv_5 i$ where $\ord(A_1)=6$, $\ord(A_3)=8$, $A_2 = A_1^2$ and  $A_4 = A_1A_3$. Thus
\[\cA=\cZ(L)=\coC[L,G_1=A_1,G_2=A_3]\]
and $\BC(\cZ[L])=(f_1,f_2,f_3)$ with $f_1=\lambda^6-\mu_1^5$, $f_2=\lambda^8-\mu_2^5$, $f_3=\mu_2^3-\mu_1^4$.
\end{example}

\section{Factorization over spectral curves}\label{sec-Ssheaves}

Given a  commutative subalgebra $\cA$ of the ring of differential operators $\Sigma [\partial]$, we review next how to define  a vector bundle $\cF_{\cA}$ on the algebraic curve $\Gamma_{\cA}$. From the extense existing literature, our point of view is closer to the results of Previato and Wilson \cite{Wilson,PW}.

In order to achieve an effective computation of $\cF_{\cA}$, we look at this problem as a factorization problem of $L-\lambda$ over a new coefficient field $\Sigma(\Gamma_{\cA})$, that contains the coefficient field $\Sigma$ of $L$ and the field of the curve $\coC (\Gamma_{\cA})$. Let us consider the ring 
\begin{equation}
\Sigma[\Gamma_{\cA}]:=\frac{\Sigma[\lambda,\overline{\mu}]}{[\BC(\cA)]}.
\end{equation}
{\bf Conjecture.} $\Sigma[\Gamma_{\cA}]$   
is a differential domain, whose ring of constants is $$\coC[\Gamma_{\cA}]\simeq \cA.$$ 
In this case, the new coefficient field $\Sigma(\Gamma_{\cA})$ for factorization of $L-\lambda$ is the fraction field of $\Sigma [\Gamma_{\cA}]$.

This conjecture was proved for $\cA=\cZ(L)$, for a Schr\" odinger operator $L$, in \cite{MRZ2}. We summarize in the following the achievements for planar and space spectral curves in \cite{RZ2024}. It would be interesting to clarify the relation between the so called spectral sheaf on the points of the spectral curve \cite{BZ} and the vector bundles that we will compute by means of greatest common right divisors in $\Sigma(\Gamma_{\cA})[\partial]$.

\subsection{Factorization over planar curves}

Given commuting differential operators $P$, $Q$ in~$\Sigma [\partial]$ and the irreducible polynomial $f\in \coC[\lambda,\mu]$ that generates the BC ideal $\BC(P,Q)$, then the differential ideal $[f]$ generated by $f$ in $\Sigma[\lambda,\mu]$ is a prime differential ideal \cite[Theorem 7]{RZ2024}, and we can consider the differential domain
\[\Sigma [\Gamma_{P,Q}]=\frac{\Sigma[\lambda,\mu]}{[f]}.\]
The derivation $\partial$ in $\Sigma$, is extended to $\Sigma [\lambda,\mu]$, setting $\partial (\lambda)=\partial(\mu)=0$.
A derivation $\tilde{\partial}$ is naturally defined on $\Sigma [\Gamma_{P,Q}]$ by $\tilde{\partial}(q+[f])=\partial(q)+[f]$, for $q\in \Sigma[\lambda,\mu]$, we denote $\tilde{\partial}$ by $\partial$ if there is no room for confusion.
We denote by $\Sigma (\Gamma_{P,Q})$ the fraction field of $\Sigma [\Gamma_{P,Q}]$, which is a differential field with the extended derivation. 

\medskip

The differential resultant of $P-\lambda$ and $Q-\mu$ is zero in $\Sigma (\Gamma_{P,Q})$, by the differential resultant theorem, see for instance \cite[Theorem A5]{RZ2024}, then  $P-\lambda$ and $Q-\mu$ have a nontrivial greatest common right divisor, as differential operators with coefficients in $\Sigma (\Gamma_{P,Q})$
\begin{equation}
    \cF(\lambda,\mu)=\gcrd(P-\lambda,Q-\mu),
\end{equation}
which is a global factor in the following sense. For every $(\lambda_0,\mu_0)\in \Gamma_{P,Q}\backslash Z$, where $Z$ is a finite subset of $\coC^2$, we obtain a right factor in $\Sigma[\partial]$ of 
 \[
L-\lambda_0 = N \cdot \cF(\lambda_0,\mu_0),
\]
where $\cF(\lambda_0,\mu_0)$ is the greatest common right divisor of $L-\lambda_0$ and $P-\mu_0$ in $\Sigma[\partial]$, the result of replacing $(\lambda, \mu)$ by $(\lambda_0,\mu_0)$. The set of points $Z$ to be removed is the finite set of points where the coefficients of $\cF$ are not well defined, see \eqref{eq-Z} for details. Observe that 
\[\ord(\cF(\lambda_0,\mu_0))\leq \ord(\cF(\lambda,\mu)), (\lambda_0,\mu_0)\in \Gamma_{P,Q}\backslash Z.\]

The next result, proved for ODOs with coefficients in $\Sigma=\coC((x))$, implies that in fact equality holds, that the dimension of the space of common solutions at each (non-singular) point of $\Gamma_{P,Q}$ is the same. A proof of Theorem \ref{thm-wilson} for an arbitrary differential field $\Sigma$ is an open question.


\begin{thm} [ G. Wilson\cite{Wilson}] \label{thm-wilson}
    Let $\Sigma=\bbC((x))$ and consider  commuting differential operators $P$, $Q$ in~$\Sigma [\partial]$. Then the differential resultant  
    \[\dres (P-\lambda,Q-\mu)=f(\lambda,\mu)^r,\] 
    for an irreducible polynomial $f$ such that $\BC(P,Q)=(f)$ and a positive integer
    \[r=\rank(\coC[P,Q])=\gcd\{\ord(T)\,|\, T\in \coC[P,Q]\}.\]
    Moreover $r=\dim (V(\lambda_0,\mu_0))$, where $V(\lambda_0,\mu_0)$ is the space of common solutions of $Py=\lambda_0 y$ and $Qy=\mu_0 y$, for any non-singular $(\lambda_0,\mu_0)$ in $\Gamma_{P,Q}$.
\end{thm}
\begin{proof}
For a rigorous proof see  \cite[Appendix]{Wilson}    
\end{proof}

\begin{cor}\label{thm-gcd}
Let $\Sigma=\bbC((x))$. Given commuting differential operators $P$, $Q$ in~$\Sigma [\partial]$, let us 
assume that $\rank (\coC[P,Q])=r$. Then the greatest common right divisor of $P-\lambda$ and $Q-\mu$ as differential operators in $\Sigma (\Gamma_{P,Q})[\partial]$ has order $r$.
\end{cor}

To give a closed formula of the greatest common right divisor of two differential operators we can use the differential subresultant sequence, see \cite{Cha,Li}.
Given differential operators $P$ and $Q$ in $\Sigma[\partial]$ of orders $n$ and $m$ respectively,
for $k=0,1,\ldots ,N:=\min\{n,m\}-1$ we define the matrix $S_k$ to be the coefficient matrix of the extended system of differential operator
\[\Xi_k=\big\{\partial^{m-1-k} (P-\lambda),\ldots, \partial (P-\lambda), P-\lambda, \partial^{n-1-k}(Q-\mu),\ldots ,\partial (Q-\mu), Q-\mu\big\}.\]
Observe that $S_k$ is a matrix with $n+m-2k$ rows, $n+m-k$ columns and entries in $\Sigma[\lambda,\mu]$.
For $i=0,\dots ,k$ let $S_k^i$ be the squared matrix of size $n+m-2k$ obtained by removing the columns of $S_k$ indexed by $\partial^{k},\ldots ,\partial,1$, except for the column indexed by~$\partial^{i}$. 
The {\sf subresultant sequence} of~$P-\lambda$ and~$Q-\mu$ is the next sequence of differential operators in $\Sigma[\lambda,\mu][\partial]$:
\begin{gather}\label{eq-subresseq}
\sdres_k (P-\lambda,Q-\mu):=\sum_{i=0}^k \phi_i \partial^i,\,\, k=0,\ldots ,N \mbox{ where } \phi_k=\det\big(S_k^i\big).
\end{gather}

By the subresultant theorem \cite[Theorem 4]{Cha}, we obtain a closed form of the greatest common right divisor of $P-\lambda$ and $Q-\mu$ that can be effectively computed.
Hence an explicit description of the rank~$r$ vector bundle $\cF$ over $\Gamma_{P,Q}$ is obtained \cite{PW}. 

\begin{cor}\label{thm-sdres}
Let $\Sigma=\bbC((x))$. Given commuting differential operators $P$, $Q$ in~$\Sigma [\partial]$, let us 
assume that $\rank (\coC[P,Q])=r$. 
Then the subresultants $\sdres_k (P-\lambda,Q-\mu)$ are identically zero in $\Sigma (\Gamma_{P,Q})[\partial]$, for $k=0,\ldots ,r-1$ and 
\[\cF(\lambda,\mu)=\sdres_r (P-\lambda,Q-\mu).\]
\end{cor}

Observe that the leading coefficient $\phi_r$ of $\sdres_r (P-\lambda,Q-\mu)$ is a polynomial in $\Sigma[\lambda,\mu]$ and the set of points where $\sdres_r (P-\lambda,Q-\mu)$ may be zero is the finite set 
\begin{equation}\label{eq-Z}
  Z=\{(\lambda_0,\mu_0)\in \bbC^2\mid \phi_r(\lambda_0,\mu_0)=0\}.  
\end{equation}

We conjecture that $\cF$ is irreducible as a differential operator with coefficients in $\Sigma(\Gamma_{P,Q})$, that is for generic values of $\lambda$ and $\mu$, and only for a finite set of points on $\Gamma_{P,Q}$ it is reducible. The next example illustrates Corollary \ref{thm-sdres} and this conjecture.

\begin{example}[Euler operators]\label{ex-EulerOp}
In $\bbC[x^{-1}][\partial]$,  the centralizer of the Euler operator 
$$L = x^{-n}\delta (\delta-m)(\delta -2m)\ldots (\delta -m(n -1))$$
where $\delta = x \partial$,  is proved in \cite{MP2023} to be $\cZ(L)=\bbC [L,B]$, where
\[B=x^{-m}\delta (\delta-n)(\delta -2n)\ldots (\delta -n(m -1)).\]
Moreover $\rank (\cZ(L))=(n,m)$ is the greatest common divisor of $n$ and $m$. It is also proved that $L^m=B^n$.    

For $n=4$ and $m=6$, let us compute the differential resultant 
\[\dres (L-\lambda,B-\mu)=(\mu^2-\lambda^3)^2.\]

The spectral curve $\Gamma_{L,B}$ of the pair $L, B$ is thus defined by $\lambda^3 - \mu^2$. We can set $\Sigma=\bbC(x)$ and consider the new coefficient field $\Sigma(\Gamma_{L,B})$, which is the quotient field of 
\[\Sigma[\Gamma_{L,B}]=\frac{\Sigma[\lambda,\mu]}{[\lambda^3 - \mu^2]}.\]

By Corollary \ref{thm-gcd} the greatest common right divisor of $L-\lambda$ and $B-\mu$ is
\[\cF(\lambda,\mu)=\sdres_2(L-\lambda,B-\mu).\]
The second differential  subresultant equals
\[\cF(\lambda,\mu)=(\lambda x^4 - 560)\left( \frac{-\mu x^2 + 20\lambda}{x^6}+\frac{-5(3\lambda x^4 - 1232)}{x^9}\partial+\frac{\lambda x^4 - 560}{x^8} \partial^2\right).\]
Observe that $\phi_2(\lambda,\mu)=\frac{(\lambda x^4 - 560)^2}{x^8}$ is non zero for every $(\lambda,\mu)$ in $\bbC^2$. The first differential subresultant equals $(\lambda x^4 - 560) (\lambda^3 - \mu^2)/x^4$, which is identically zero in $\Sigma(\Gamma_{P,Q})[\partial]$.

At the singular point $\cF(0,0)=(\partial-\frac{11}{x})\partial$, while $\cF(1,1)$ is irreducible. We conjecture that $\cF(\lambda_0,\mu_0)$ will be irreducible for every non singular $\lambda_0,\mu_0$ of $\Gamma_{L,B}$. 
\end{example}

\medskip

Let us go back to differential operators with coefficients in an arbitrary differential field $\Sigma$.
If $\rank (\coC[P,Q])=1$ then it can be proved as \cite[Lemma 8]{RZ2024} that $\cF$ is the order one differential operator determined by the first differential subresultant $\sdres_1(P-\lambda,Q-\mu)=\phi_0+\phi_1 \partial$. 

\begin{lem}\label{lem-phi}
If $\rank(\coC[P,Q])=1$ then $\phi_{i}\notin [f]$. 
\end{lem}
\begin{proof}
By the construction of the matrices $S_1^i$ as in \eqref{eq-subresseq}, the degree in $\lambda$ of $f$ is $\ord(Q)=m$ and the degree in $\mu$ is $\ord(P)=n$. By the construction of the matrices $S_1^i$ in \eqref{eq-subresseq}, the degree in $\mu$ of $\phi_{i}(\lambda,\mu)$ is less than or equal to $n-1$. Therefore $\phi_{i}$ does not belong to $[\BC(P,Q)]=[f]$. 
\end{proof}

\begin{thm}\label{thm-SSPlanar}
Let $\Sigma$ be a differential field whose field of constants $\coC$ is algebraically closed and has zero characteristic. Given commuting differential operators $P$ and $Q$ in $\Sigma [\partial]$, let $\BC(P,Q)=(f)$. If $\rank (P,Q)=1$ then $\cF=\sdres_1(P-\lambda,Q-\mu)$ is the greatest common right divisor over $\Sigma(\Gamma_{P,Q})$ of $P-\lambda$ and $Q-\mu$.
\end{thm}
\begin{proof}
    By Lemma \ref{lem-phi}, $\sdres_1(P-\lambda,Q-\mu)$ is an order one operator and by the first subresultant theorem \cite[Theorem B.4]{RZ2024} it is the greatest common divisor $\cF$ of $P-\lambda$ and $Q-\mu$. 
\end{proof}

In some cases centralizers are isomorphic to coordinate rings of planar spectral curves. For instance, if $L$ is a differential operator of order $2$ then $\cA=\cZ (L)=\coC [L,A]$ with $\ord(A)\equiv_2 1$. See also the family of Euler operators in Example \ref{ex-EulerOp}.

\subsection{Factorization over space curves}

Let us assume that $\cA=\cZ(L)$ for an operator $L$ of order $3$ and whose basis as a $\coC [L]$-module is $\{1,A_1,A_2\}$. Then by \cite[Theorem 10]{RZ2024} $\BC(\cA)=(f_1,f_2,f_3)$ and by \cite[Theorem 11]{RZ2024} the ideal $[\BC(\cA)]$ generated by $\BC(\cA)$ in $\Sigma[\lambda,\mu_1,\mu_2]$ is a prime differential ideal. This allows to define the differential domain 
\begin{equation}
\Sigma[\Gamma_{\cA}]=\frac{\Sigma[\lambda,\mu_1,\mu_2]}{[\BC(\cA)]},
\end{equation}
with the natural derivation and extend this derivation to its fraction field $ \Sigma(\Gamma_{\cA})$.

\begin{thm}\label{thm-SSSpatial}
Let us assume that $\cA=\cZ(L)$ for an operator $L$ of \textcolor{teal}{order $3$} whose basis as a $\coC [L]$-module is $\{1,A_1,A_2\}$. 
Then the monic greatest common right divisor 
$\cF$ over $\Sigma(\Gamma_{\cA})$ 
of $L-\lambda$, $A_1-\mu_1$ and $A_2-\mu_2$, equals the monic $\gcrd(L-\lambda,A_i-\mu_i)$, $i=1,2$.
\end{thm}
\begin{proof}
In this situation \cite[Proposition 1 and Lemma 9]{RZ2024} are satisfied and
\cite[Theorem 12]{RZ2024}  applies to conclude that $$\cF=\partial+\phi, \mbox{ with }\phi=\frac{\phi_0^i}{\phi_1^i}+[\BC(\cA)]$$
a monic order one operator obtained by computing the first differential subresultants
\[\sdres_1(L-\lambda,A_i-\mu)=\phi_0^i+\phi_1^i \partial.\]
\end{proof}

\begin{example}\label{ex-3gen}
 Consider the differential field $\Sigma=\coC \langle\cosh(x)\rangle$ with derivation $\partial=d/dx$ and 
\[L=\frac{-12 \sinh(x)}{\cosh(x)^3} +\frac{6}{\cosh(x)^2} \partial + \partial^3.\]
For convenience we denote $\eta=\cosh(x)$.
The centralizer equals $$\cA=\cZ(L)=\coC[L,A_1,A_2]$$ with
\begin{align*}
A_1=& \frac{16\eta^2 - 24}{\eta^4} + \frac{-24 \eta'}{\eta^3}\partial + \frac{(-4/3)\eta^2 + 8}{\eta^2}\partial^2 + \partial^4,
\\
A_2=&
\frac{(-160/3) \eta^2 + 80}{\eta^5} \eta' +\frac{(16/9)\eta^4 + (200/3) \eta^2 - 80}{\eta^4}  \partial\\
&+ \frac{-40\eta'}{\eta^3} \partial^2 + \frac{10}{\eta^2} \partial^3 + \partial^5.
\end{align*}
The differential resultants, denoted as in Section \ref{sec-SpaceSC}, are
\begin{align*}
f_1(\lambda,\mu_1)=&\lambda^4- \mu_1^3-4\lambda^2\mu_1  - \frac{64}{27}\lambda^2, \\
f_2(\lambda,\mu_2)= &\lambda^5 - \mu_2^3+ \frac{16}{3}\mu_2\lambda^2 +\frac{4096}{729}\lambda,\\
f_3(\mu_1,\mu_2)=&\mu_2^4-\mu_1^5 - \frac{20}{3}\mu_2^2\mu_1^2 -\frac{64}{9} \mu_1^4 - \frac{704}{27}\mu_2^2\mu_1 - \frac{2048}{81}\mu_1^3 - \frac{4096}{243}\mu_2^2\\
&- \frac{32768}{729}\mu_1^2 - \frac{262144}{6561}\mu_2.
\end{align*}
By Theorem \ref{thm-BCL} then $\BC(\cZ(L))=(f_1,f_2,f_3)$.
By Theorem \ref{thm-SSSpatial} the greatest common right divisor of $L-\lambda$, $A_1-\mu_1$ and $A_2-\mu_2$ is $\cF=\partial-\phi$ with
\[\phi=\frac{-59049  \cosh(x)^2 s^3 + 4374  \sinh(x) s^2  \cosh(x) - 108  \cosh(x)^2 s - 12 \tanh(x) + 324 s}{2187  \cosh(x)^2 s^2 - 162  \cosh(x)  \sinh(x) s + 4  \cosh(x)^2 - 6}\]
obtained by
replacing $(\lambda,\mu_1,\mu_2)$ by a parametrization 
\[\chi(\tau)=(19683 s^3, 531441 s^4 - 972 s^2, 14348907s^5 + 48s)\] 
of $\Gamma_{\cA}$ in 
$\sdres(L-\lambda,A_i-\mu_i)$, $i=1,2$.
The common factor $\cF$ defines the line bundle on the spectral curve $\Gamma_{\cA}$.
\end{example}

\section{Spectral Picard-Vessiot fields}\label{sec-PV}

Following Drach's ideology \cite{Drach3}, we aim to compute the minimal field containing the solutions of the spectral problem $(L-\lambda)(y)=0$. 
Consider a finitely generated commutative $\coC$-algebra $\cA$ in $\Sigma [\partial]$ as in \eqref{eq-cA} and the field of the curve $\Gamma_{\cA}$, the fraction field $\coC(\Gamma_{\cA})$ of the domain 
\[\coC [\Gamma_{\cA}]=\frac{\coC [\lambda,\overline{\mu}]}{\BC(\cA)}\simeq \cA.\] 
It is important to note that $\coC(\Gamma_{\cA})$ is not an algebraically closed field, which is required to prove existence of classical Picard-Vessiot fields \cite{VPS}. Differential Galois theories do exist in the case of non algebraically closed field of constants, \cite{CrespoH},   \cite{MO2018} but, to fully exploit its special properties, we aim to develop a Picard-Vessiot theory adapted to spectral problems $L(y)=\lambda y$, for a generic algebraic parameter $\lambda$ over the field of constants $\coC$.

\medskip

Let $\cA=\cZ(L)=\coC[L,A]$ be the centralizer of a Schr\" odinger operator $L=-\partial^2+u$, $u\in\Sigma$. Denote by $\partial-\phi_+$, resp. $\partial-\phi_-$, the monic greatest common right divisor of $L-\lambda$ and $A-\mu$, resp. $A+\mu$. Let $\Psi_+$, resp. $\Psi_-$, be a non zero solution of $(L-\lambda)(y)=0$ defined by 
\[\partial(\Psi_+)=\phi_{+}\Psi_+, \mbox{ resp. } \partial(\Psi_-)=\phi_{-}\Psi_-.\]
The solutions $\Psi_+$ and $\Psi_-$ belong to the differential closure of $\Sigma(\Gamma_{\cA})$ and by \cite[Lemma 4.5]{MRZ2} are a fundamental system of 
 solutions of $(L-\lambda)(y)=0$.

\begin{thm}[Existence of spectral Picard-Vessiot fields \cite{MRZ2}]
Let $\cA$ be the centralizer of a Schr\" odinger operator $L=-\partial^2+u$, where $u\in \Sigma$. 
The following statements hold:
\begin{enumerate}
\item The field of constants of $(\Sigma (\Gamma_{\cA}),\tilde{\partial})$ is $\coC(\Gamma_{\cA})$.
\item  The field of constants of the differential field extension
\begin{equation}
    \cE(L-\lambda)=\Sigma (\Gamma_{\cA})\langle \Psi_+\rangle=\Sigma (\Gamma_{\cA})\langle \Psi_+,\Psi_{-}\rangle,
\end{equation}
of $\Sigma (\Gamma_{\cA})$ is $\coC(\Gamma_{\cA})$.
\end{enumerate}
We call $\cE(L-\lambda)$ a {\sf spectral Picard-Vessiot (PV) field over $\Gamma_{\cA}$ of} $(L-\lambda)(y)=0$.
\end{thm}

At all but a finite number of points $(\lambda_0,\mu_0)$ of the spectral curve, the spectral PV field $\cE(L-\lambda)$ specializes to the PV field of the equation $(L-\lambda_0)(y)=0$, a differential field extension of $\Sigma$ \cite{MRZ2}. Classical differential Galois groups of spectral problems $L(y)=\lambda_0 y$, $\lambda_0\in \coC$ were studied in \cite{Grig} for algebro-geometric operators, in \cite{Brez} for Schr\" odinger operators and in \cite{AMW}, for Schr\" odinger operators not restricted to the algebro-geometric case. 

\medskip

In the case  $\cA=\cZ(L)$ with $\ord(L)=3$, we conjecture that $\coC(\Gamma_{\cA})$ is the field of constants of $\Sigma (\Gamma_{\cA})$. As a continuation of the results in \cite{RZ2024}, we are working to prove the existence of a differential field extension 
\[\cE(L-\lambda)= \Sigma (\Gamma_{\cA}) \langle \Psi_1,\Psi_2,\Psi_3 \rangle \]
of $\Sigma (\Gamma_{\cA})$ by a fundamental system of solutions of  $(L-\lambda)(y)=0$, and that the field of constants of $\cE(L-\lambda)$ is $\coC(\Gamma_{\cA})$. 

\medskip

For $\ord(L)>3$, it remains to prove that the differential ideal $[\BC(\cA)]$ is a prime ideal in $\Sigma [\lambda,\overline{\mu}]$, before defining the new coefficient field $\Sigma (\Gamma_{\cA})$ and  computing the right factors of $L-\lambda$ over $\Sigma (\Gamma_{\cA})$. Our aim is to prove the existence of spectral Picard-Vessiot fields for ODOs with non-trivial centralizers of rank $1$. 

\medskip

Determining if a commutative algebra $\cA$ of ODOs has rank $1$ is a non trivial open problem to which differential Galois theory brings a new perspective. The differential Galois group of an ordinary differential operator $L-\lambda$ was defined in \cite{BEG}, and the centralizer of $L$ was proved to have rank $1$ if the group is commutative. More precisely in \cite{BEG} they define the differential Galois group of a  quantum complete integrable system (QCIS). It is important to note, that in the definition of differential Galois group given in \cite{BEG} the parameter $\lambda$ is a generic parameter, a generic constant with respect to the derivation. It would be interesting to understand its relation with the differential Galois group of the parametric Picard-Vessiot theory introduced by Cassidy and Singer in \cite{CS}, or appropriate extensions of this theory \cite{GGO2013}. 

Other authors have studied the differential Galois group of spectral problems $Ly=\lambda_0 y$, with $\lambda=\lambda_0$ an specialization of the spectral parameter \cite{Brez3, Grig, AMW}, by means of the differential Galois group of the classical Picard-Vessiot theory, where the field of constants is algebraically closed. Our aim is to develop algorithms to compute  differential Galois groups with parameters of spectral problems $Ly=\lambda y$, in the same spirit of the existing ones for differential Galois groups \cite{singer1996testing}, \cite{Hrushovski2002}, \cite{Feng2015} or parametrized differential Galois groups \cite{Arr}, \cite{MO2018}, to determine if the centralizer of $L$ has rank $1$. 

\medskip

\noindent{\bf Acknowledgments.} Supported by the grant PID2021-124473NB-I00, ``Algorithmic Differential Algebra and Integrability" (ADAI)  from the Spanish MCIN/AEI /10.13039/501100011033 and by FEDER, UE. The author is grateful to M.A. Zurro for numerous stimulating discussions, to A. Jiménez Pastor for computations of centralizers with the \texttt{dalgebra} package in SageMath and to the anonymous referee whose remarks allowed to improve the exposition. 

\bibliographystyle{plain}

\bibliography{Bibliography.bib}

\begin{thebibliography}{10}

\bibitem{AMW}
P.B. Acosta-Hum\'anez, J.J Morales-Ruiz, and J.A. Weil.
\newblock {Galoisian approach to integrability of Schr\"odinger equation}.
\newblock {\em Rep. Math. Phys.}, 67(3):305--374, 2011.

\bibitem{Arr}
C.~Arreche.
\newblock {On the computation of the parameterized differential Galois group for a second-order linear differential equation with differential parameters}.
\newblock {\em J. Symbolic Comput.}, 75:25--55, 2016.

\bibitem{BEG}
A.~Braveman, P.~Etingof, and D.~Gaitsgory.
\newblock Quantum integrable systems and differential {G}alois theory.
\newblock {\em Transformation Groups}, 2:31--56, 1997.

\bibitem{Brez3}
Y.V. Brezhnev.
\newblock {Spectral/quadrature duality: Picard-Vessiot theory and finite-gap potentials}.
\newblock {\em Contemp. Math.}, 563(1):p.1, 2012.

\bibitem{Brez}
Y.V. Brezhnev.
\newblock {Elliptic solitons, Fuchsian equations, and algorithms}.
\newblock {\em St. Petersburg Math. J.}, 24(4):555--574, 2013.

\bibitem{BZ}
I.~Burban and A.~Zheglov.
\newblock Fourier-{M}ukai transform on {W}eierstrass cubics and commuting differential operators.
\newblock {\em Internat. J. Math.}, 29(10):1850064,46pp, 2018.

\bibitem{BC1}
J.L. Burchnall and T.W. Chaundy.
\newblock Commutative ordinary differential operators.
\newblock {\em Proc. Lond. Math. Soc.}, s2-21:420--440, 1923.

\bibitem{BC2}
J.L. Burchnall and T.W. Chaundy.
\newblock Commutative ordinary differential operators {II}. the {I}dentity ${P}^n = {Q}^m$.
\newblock {\em Proc. R. Soc. Lond.}, 134:471--485, 1931.

\bibitem{CS}
P.J. Cassidy and M.F. Singer.
\newblock {Galois theory of parameterized differentialequations and linear differential algebraic groups}.
\newblock {\em Differential Equations and Quantum Groups (IRMA Lect. Math. Theor. Phys.)}, 9:113-- 157, 1991.

\bibitem{Cha}
M.~Chardin.
\newblock {Differential resultants and subresultants}.
\newblock {\em Proc. FCT'91. Lecture Notes in Computer Science. Springer-Verlag}, 529:471--485, 1991.

\bibitem{CrespoH}
T.~Crespo and Z.~Hajto.
\newblock {\em {Algebraic Groups and Differential Galois Theory}}, volume 122 of {\em Graduate Studies in Mathematics}.
\newblock Amer. Math. Soc., Providence, Rhode Island, 2011.

\bibitem{DM}
V.~N. Davletshina and A.~E. Mironov.
\newblock On commuting ordinary differential operators with polynomial coefficients corresponding to spectral curves of genus two.
\newblock {\em Bull. of the Korean Math. Soc.}, 54(5):1669--1675, 09 2017.

\bibitem{DS}
V.~N. Davletshina and E.I. Shamaev.
\newblock On commuting differential operators of rank 2.
\newblock {\em Sib Math J}, 55(4):606--610, 2014.

\bibitem{Dikii}
L.~A. Dickey.
\newblock {\em {Soliton equations and Hamiltonian systems}}, volume~26.
\newblock World Scientific, 2003.

\bibitem{Drach3}
J.~Drach.
\newblock Sur l'int\'egration par quadrature de l'\'equation $d^2y/dx^2= [\phi(x) + h] y$.
\newblock {\em C. R. Acad. Sci. Paris}, 168:337--340, 1919.

\bibitem{DS85}
V.~G. Drinfel'd and V.V. Sokolov.
\newblock {Lie algebras and equations of KdV type}.
\newblock {\em Soviet J. Math}, 30:1975--2036, 1985.

\bibitem{Feng2015}
Ruyong Feng.
\newblock {Hrushovski’s algorithm for computing the Galois group of a linear differential equation}.
\newblock {\em Advances in Applied Mathematics}, 65:1–37, April 2015.

\bibitem{GGO2013}
Henri Gillet, Sergey Gorchinskiy, and Alexey Ovchinnikov.
\newblock {Parameterized Picard–Vessiot extensions and Atiyah extensions}.
\newblock {\em Advances in Mathematics}, 238:322–411, May 2013.

\bibitem{Good}
K.R. Goodearl.
\newblock Centralizers in differential, pseudo-differential and fractional differential operator rings.
\newblock {\em Rocky Mountain J. Math.}, 13(4):573--618, 1983.

\bibitem{Grig}
N.V. Grigorenko.
\newblock {Algebraic-geometric operators and Galois differential theory}.
\newblock {\em Ukrainian Math. J.}, 61:14--29, 2009.

\bibitem{Gri}
P.G. Grinevich.
\newblock Rational solutions for the equation of commutation of differential operators.
\newblock {\em Funct. Anal. Appl.}, 16(1):15--19, 1982.

\bibitem{GZ24}
J.~Guo and A.~Zheglov.
\newblock {On some questions around Berest’s conjecture}.
\newblock {\em Mathematical Notes}, 116(1):238--251, 2024.

\bibitem{Hrushovski2002}
Ehud Hrushovski.
\newblock {Computing the Galois group of a linear differential equation}.
\newblock In {\em Differential Galois Theory}, page 97–138. Institute of Mathematics Polish Academy of Sciences, 2002.

\bibitem{GD}
{I. M. Gel'fand, and L. A. Dikii}.
\newblock {Asymptotic behaviour of the resolvent of Sturm-Liouville equations and the algebra of the Korteweg-de Vries equations}.
\newblock {\em Russian Math. Surveys}, 30(5):77--113, 1975.

\bibitem{JR2025}
A.~Jim\'{e}nez-Pastor and S.~L. Rueda.
\newblock Effective computation of centralizers.
\newblock {\em arXiv: 2505.01289.}, 2025.

\bibitem{JRZHD2025}
Antonio Jimenez-Pastor, Sonia~L. Rueda, Maria-Angeles Zurro, Rafael Hernandez~Heredero, and Rafael Delgado.
\newblock {Computing almost commuting bases of ODOs and Gelfand-Dickey hierarchies}.
\newblock {\em Mathematics in Computer Science}, 19(1), May 2025.

\bibitem{KasmanPreviato2001}
A.~Kasman and E.~Previato.
\newblock Commutative partial differential operators.
\newblock {\em Physica D: Nonlinear Phenomena}, 152-153:66--77, 2001.
\newblock Advances in Nonlinear Mathematics and Science: A Special Issue to Honor Vladimir Zakharov.

\bibitem{Kolchin}
E.~R. Kolchin.
\newblock {\em Differential algebra and algebraic groups}.
\newblock Number~54 in Pure and Applied Mathematics. Academic Press, Boston, MA, 1973.

\bibitem{K77bis}
I.~M. Krichever.
\newblock Methods of algebraic geometry in the theory of nonlinear equations.
\newblock {\em Uspehi Mat. Nauk}, 32(6):183--208, 287, 1977.

\bibitem{K78}
I.~M. Krichever.
\newblock Commutative rings of ordinary linear differential operators.
\newblock {\em Funct. Anal. Appl.}, 12(3):175--185, 1978.

\bibitem{KN2}
I.~M. Krichever and S.~P. Novikov.
\newblock Holomorphic bundles over algebraic curves and nonlinear equations.
\newblock {\em Russian Math. Surveys}, 35(6):53--79, 1980.

\bibitem{Li}
Z.~Li.
\newblock A subresultant theory for {O}re polynomials with applications.
\newblock {\em Proc. Int. Symp. Symbolic and Algebraic Computation}, pages 132--139, 1998.

\bibitem{MakarLimanov2021}
L.~Makar-Limanov.
\newblock Centralizers of rank one in the first weyl algebra.
\newblock {\em Symmetry, Integrability and Geometry: Methods and Applications}, May 2021.

\bibitem{MP2023}
L.~Makar-Limanov and E.~Previato.
\newblock Centralizers of differential operators of rank h.
\newblock {\em Journal of Geometry and Physics}, 188:104812, May 2023.

\bibitem{McW}
S.~McCallum and F.~Winkler.
\newblock Resultants: Algebraic and differential.
\newblock {\em Techn. Rep. J. Kepler University}, RISC18-08:pp. 21, 2018.

\bibitem{MO2018}
A.~Minchenko and A.~Ovchinnikov.
\newblock {Calculating Galois groups of third-order linear differential equations with parameters}.
\newblock {\em Communications in Contemporary Mathematics}, 20(4):1750038, 2018.

\bibitem{Mi1}
A.~E. Mironov.
\newblock Self-adjoint commuting ordinary differential operators.
\newblock {\em Invent. Math.}, 197(2):417--431, 2014.

\bibitem{Mi2}
A.E. Mironov and A.B. Zheglov.
\newblock Commuting ordinary differential operators with polynomial coefficients and automorphisms of the first{ W}eyl algebra.
\newblock {\em Int. Math. Res. Not. IMRN}, 2016(10):2974--2993, 2016.

\bibitem{Mo2}
O.I. Mokhov.
\newblock Commuting ordinary differential operators of arbitrary genus and arbitrary rank with polynomial coefficients.
\newblock {\em Topology, Geometry, Integrable Systems, and Mathematical Physics: Novikov's Seminar 2012-2014}, 234(S. 2):323--336, 2014.

\bibitem{Morales}
J.~J. Morales-Ruiz.
\newblock {\em {Differential Galois theory and non-integrability of Hamiltonian systems}}.
\newblock Birkh\"auser, Berlin, 1999.

\bibitem{MRZ1}
J.~J. Morales-Ruiz, S.L. Rueda, and M.A. Zurro.
\newblock {Factorization of KdV Schr\" odinger operators using differential subresultants}.
\newblock {\em Adv. Appl. Math.}, 120:102065, 2020.

\bibitem{MRZ2}
J.~J. Morales-Ruiz, S.L. Rueda, and M.A. Zurro.
\newblock {Spectral {P}icard–{V}essiot fields for algebro-geometric {S}chr\" odinger operators}.
\newblock {\em Annales de l'Institut Fourier}, 71(3):1287--1324, 2021.

\bibitem{Mulase}
M.~Mulase.
\newblock Category of vector bundles on algebraic curves and infinite dimensional grassmannians.
\newblock {\em International Journal of Mathematics}, 1(3):293--342, 1990.

\bibitem{Mum2}
D.~Mumford.
\newblock An algebro-geometric construction of commuting operators and of solutions to the {T}oda lattice equation, {K}orteweg de {V}ries {E}quation and related non-linear equations.
\newblock {\em Proc. of the Int. Symp. on Alg. Geom., (Kyoto Univ., Kyoto, 1977)}, pages 115--153, 1978.

\bibitem{Og2016}
V.~S. Oganesyan.
\newblock Commuting differential operators of rank 2 with polynomial coefficients.
\newblock {\em Funct. Anal. Appl.}, 50(1):54–61, 2016.

\bibitem{PZ}
D.~A. Pogorelov and A.~B. Zheglov.
\newblock An algorithm for construction of commuting ordinary differential operators by geometric data.
\newblock {\em Lobachevskii Journal of Mathematics}, 38(6):1075--1092, 2017.

\bibitem{Prev}
E.~Previato.
\newblock Another algebraic proof of {W}eil's reciprocity.
\newblock {\em Atti Accad. Naz. Lincei Cl. Sci. Fis. Mat. Natur. Rend. Lincei (9) Mat. Appl}, 2(2):167--171, 1991.

\bibitem{PRZ2019}
E.~Previato, S.~L. Rueda, and M.~A. Zurro.
\newblock Commuting ordinary differential operators and the {D}ixmier test.
\newblock {\em {SIGMA Symmetry Integrability Geom. Methods Appl.}}, 15(101):23 pp., 2019.

\bibitem{PRZ2023}
E.~Previato, S.~L. Rueda, and M.~A. Zurro.
\newblock {Burchnall–Chaundy polynomials for matrix ODOs and Picard–Vessiot Theory}.
\newblock {\em {Physica D: Nonlinear Phenomena}}, 453:133811, 2023.

\bibitem{PW}
E.~Previato and G.~Wilson.
\newblock Differential operators and rank $2$ bundles.
\newblock {\em Compos. Math.}, 81(1):107--119, 1992.

\bibitem{VPS}
M.~van~der Put and M.~F. Singer.
\newblock {\em Galois theory of linear differential equations}, volume 328 of {\em {Grundlehren der Mathematischen Wissenschaften}}.
\newblock Springer Science \& Business Media, 2012.

\bibitem{Ritt}
J.F. Ritt.
\newblock {\em Differential algebra}, volume~33.
\newblock American Mathematical Soc., 1950.

\bibitem{RZ2024}
S.L. Rueda and M.A. Zurro.
\newblock {Spectral Curves for Third-Order ODOs}.
\newblock {\em Axioms}, 13(4):274, 2024.

\bibitem{Sch}
I.~Schur.
\newblock \"{U}ber vertauschbare lineare {D}ifferentialausdr\"ucke.
\newblock {\em Berlin Math. Gesellschaft, Sitzungsbericht. Arch. der Math., Beilage}, 3(8):2--8, 1904.

\bibitem{SW}
G.B. Segal and G.~Wilson.
\newblock Loop groups and equations of {K}d{V} type.
\newblock {\em Publ. Math. Inst. Hautes \'Etudes Sci.}, 61:5--65, 1985.

\bibitem{singer1996testing}
M.~F. Singer.
\newblock Testing reducibility of linear differential operators: a group theoretic perspective.
\newblock {\em Applicable Algebra in Engineering, Communication and Computing}, 7(2):77--104, 1996.

\bibitem{SAGE2023}
{The Sage Developers}.
\newblock {\em {S}ageMath, the {S}age {M}athematics {S}oftware {S}ystem ({V}ersion 9.7)}, 2023.
\newblock {\tt https://www.sagemath.org}.

\bibitem{Ve}
J.~L. Verdier.
\newblock \'{E}quations diff\' erentielles algebriques.
\newblock {\em Lecture Notes in Math.}, 710:101--122, 1979.

\bibitem{We}
R.~Weikard.
\newblock On commuting differential operators.
\newblock {\em Electron. J. Differential Equations}, 2000(19):1--11, 2000.

\bibitem{Wilson81}
G.~Wilson.
\newblock {The modified Lax and two-dimensional Toda lattice equations associated with simple Lie algebras}.
\newblock In {\em Ergodic Theory and Dynamical Systems}, volume~1, pages 361--380. 1981.

\bibitem{Wilson}
G.~Wilson.
\newblock Algebraic curves and soliton equations.
\newblock In {\em Geometry {T}oday}, volume~60 of {\em Progr. Math.}, pages 303--329. Birkh\"auser, Boston, 1985.
\newblock E. Arbarello et al. ed.

\bibitem{Z2024}
A.~B. Zheglov.
\newblock {The Conjecture of Dixmier for the first Weyl algebra is true.}
\newblock {\em arXiv}, page 2410.06959, 2024.

\bibitem{Zheglov}
A.B. Zheglov.
\newblock {\em Algebra, Geometry and Analysis of Commuting Ordinary Differential Operators.}
\newblock {Moscow State University, Moscow}, 2020.

\bibitem{Zheglov2020}
Alexander Zheglov.
\newblock {\em Algebraic Geometric Properties of Spectral Surfaces of Quantum Integrable Systems and Their Isospectral Deformations}, page 313–331.
\newblock Springer International Publishing, 2020.

\end{thebibliography}

\end{document}